\documentclass[11pt,a4paper]{amsart}
\usepackage[english]{babel}
\usepackage[T1]{fontenc}
\usepackage{mathpazo}  
\usepackage{upref}
\usepackage{enumerate}
\usepackage[colorlinks=true, pdftitle={A Liouville theorem for
  p-harmonic functions on exterior domains}, pdfauthor={E.N. Dancer,
  Daniel Daners and Daniel Hauer}]{hyperref}

\title{A Liouville theorem for $p$-harmonic functions on exterior
  domains}

\author{E. N. Dancer}
\author{Daniel Daners}
\author{Daniel Hauer}
\address{School of Mathematics and Statistics,
  The University of Sydney, NSW 2006, Australia}
\email[E. N. Dancer]{\href{mailto:norman.dancer@sydney.edu.au}{
\nolinkurl{norman.dancer@sydney.edu.au}}}
\email[Daniel Daners]{\href{mailto:daniel.daners@sydney.edu.au}{
\nolinkurl{daniel.daners@sydney.edu.au}}}
\email[Daniel Hauer]{\href{mailto:daniel.hauer@sydney.edu.au}{
\nolinkurl{daniel.hauer@sydney.edu.au}}}

\date{16. November 2014 (revised Version)}



\numberwithin{equation}{section}

\theoremstyle{theorem}
\newtheorem{theorem}{Theorem}[section]
\newtheorem{proposition}[theorem]{Proposition}
\newtheorem{lemma}[theorem]{Lemma}

\theoremstyle{definition}
\newtheorem{definition}[theorem]{Definition}
\newtheorem{assumption}[theorem]{Assumption}

\theoremstyle{remark}
\newtheorem{remark}[theorem]{Remark}

\newcommand\R{{\mathbb{R}}}

\newcommand\dx{\mathrm{d}x }

\newcommand\dr{\mathrm{d}r }
\newcommand\ds{\mathrm{d}s }

\newcommand\dH{\mathrm{d}\mathcal{H}}

\DeclareMathOperator*{\divergence}{div}
\DeclareMathOperator{\supp}{supp}
\DeclareMathOperator{\loc}{loc}

\newcommand\abs[1]{\lvert#1\rvert} 
\newcommand\norm[1]{\lVert#1\rVert} 

\begin{document}

\begin{abstract}
  We prove Liouville type theorems for $p$-harmonic functions on
  exterior domains of $\R^{d}$, where $1<p<\infty$ and $d\geq 2$. We
  show that every positive $p$-harmonic function satisfying zero
  Dirichlet, Neumann or Robin boundary conditions and having zero limit
  as $|x|$ tends to infinity is identically zero. In the case of zero
  Neumann boundary conditions, we establish that any semi-bounded
  $p$-harmonic function is constant if $1<p<d$. If $p\ge d$, then it is
  either constant or it behaves asymptotically like the fundamental
  solution of the homogeneous $p$-Laplace equation.
\end{abstract}

\maketitle

\section{Introduction and main results}
Assume that $\Omega$ is a general \emph{exterior domain} of $\R^{d}$,
that is, a connected open set such that $\Omega^{c}=\R^{d}\setminus
\Omega$ is compact and nonempty. We assume that the boundary
$\partial\Omega$ is the disjoint union of the sets $\Gamma_{1}$,
$\Gamma_{2}$, where $\Gamma_1$ is closed. We denote by $\nu$ the outward
pointing unit normal vector on $\partial\Omega$ and $\mathcal{H}$ the
$(d-1)$-dimensional Hausdorff measure on $\partial\Omega$. For
$1<p<\infty$ define the $p$-Laplace operator $\Delta_p$ by
$\Delta_{p}v:=\divergence(\abs{\nabla v}^{p-2}\nabla v)$.

The aim of this paper is to establish a Liouville theorem for weak
solutions of the elliptic boundary-value problem
\begin{equation}\label{eq:1}
  \begin{aligned}
    -\Delta_{p} v &= 0 && \text{in $\Omega$,}\\
    \mathcal{B}v&=0 && \text{on $\partial \Omega$,}
  \end{aligned}
\end{equation}
where
\begin{displaymath}
  \mathcal{B}v:=
  \begin{cases}
    v|_{\Gamma_{1}}
    &\text{on $\Gamma_{1}$ (Dirichlet b.c.),}\\
    \abs{\nabla v}^{p-2}\frac{\partial v}{\partial \nu}+h(x,v)
    &\text{on $\Gamma_{2}$ (Robin/Neumann b.c.).}\\
  \end{cases}
\end{displaymath}
Here we assume that $h\colon\Gamma_2\times \R\to \R$ is a Carath\'eodory
function (see \cite{MR0176364}) satisfying
\begin{equation}\label{assum:h}
  h(\cdot,v)\in L^{p/(p-1)}(\Gamma_{2})
  \quad\text{and}\quad
  \text{$h(x,v)v\geq 0$ for $\mathcal{H}$-a.e. $x\in \Gamma_{2}$,}
\end{equation}
for every $v\in L^{p}(\Gamma_2)$. Note that the first condition
in~\eqref{assum:h} implicitly implies a growth condition on the function
$v\mapsto h(\cdot,v)$; see \cite{MR0159197}. As usual, a function $v\in
W^{1,p}_{\loc}(\Omega)\cap C(\Omega)$ is said to be \emph{$p$-harmonic}
(or simply \emph{harmonic} if $p=2$) on $\Omega$ if $\Delta_{p}v=0$ in
$\Omega$ in the weak sense, that is,
\begin{displaymath}
  \int_{\Omega}\abs{\nabla v}^{p-2}\nabla v\nabla \varphi\,\dx=0
\end{displaymath}
for every $\varphi\in C^{\infty}_{c}(\Omega)$; see
\cite{MR0170096}. Throughout, we call a $p$-harmonic function $v$ \emph{positive}
if $v\geq 0$.

The classical Liouville theorem asserts that every harmonic function on
the whole space $\R^{d}$ is constant if it is bounded from below or from
above; see for instance \cite[Theorem 3.1]{MR1805196} or
\cite[p. 111]{MR762825}.  The classical Liouville theorem was
generalised to $p$-harmonic functions on the whole space $\R^{d}$ for
$1<p<\infty$; see \cite[Theorem II]{MR1946918} or \cite[Corollary
6.11]{MR1207810}. The result extends to $d$-harmonic function on
$\R^{d}\setminus\{0\}$ for $d\ge 2$; see \cite[Corollary 3.3]{MR1805196}
for $p=d=2$ or \cite[Corollary 2.2]{MR859333} for $p=d\geq 2$. In our
investigation, the \emph{fundamental solution}
\begin{equation}
  \label{eq:2}
  \mu_{p}(x):=
  \begin{cases}
    \abs{x}^{(p-d)/(p-1)} & \text{if $p\neq d$,}\\
    \log \abs{x} & \text{if $p=d$.}
  \end{cases}
\end{equation}
on $\R^{d}\setminus\{0\}$ plays an important role.  If $1<p<d$, then
$\mu_p$ provides an example of a non-constant $p$-harmonic function
bounded from below. Another example valid for $1<p< d$ is given by
$v(x):=1-\mu_{p}(x)$ for every $x\in \mathbb R^d\setminus B_1$, where
$B_{1}$ denotes the open unit ball. In this case $v$ is a positive
$p$-harmonic function on the exterior domain
$\Omega:=\R^{d}\setminus\overline{B}_{1}$ satisfying zero Dirichlet
boundary conditions at $\partial \Omega$. Hence, in order to have a
chance of proving a Liouville type theorem for exterior domains we need
to make use of the boundary conditions and the behavior of a
$p$-harmonic function near infinity.

First, we consider the case $1<p<d$. Then by \cite[Corollary,
p.84]{MR0186903} or~\cite[Theorem 2 \& Theorem 3]{MR2424533} and by
rescaling if necessary, we know that for every positive $p$-harmonic
function $v$ on an exterior domain $\Omega\subseteq \R^{d}$, the limit
\begin{equation}
  \label{eq:6}
  b:=\lim_{|x|\to\infty}v(x)\quad\text{exists and}\quad
  |v(x)-b|\le c_{1}\,\mu_{p}(x)\text{ whenever $|x|\ge 2$}
\end{equation}
where $c_{1}>0$. With this in mind, our first result is a kind of
maximum principle for weak solution of~\eqref{eq:1} on an unbounded
domain. A precise definition of weak solutions of \eqref{eq:1} is given
in Definition~\ref{def:wsol} below.
\begin{theorem}
  \label{thm:1}
  Let $\Omega$ be an exterior domain with Lipschitz boundary and let
  $1<p<\infty$. Suppose that \eqref{assum:h} is satisfied and that $v$ is a
  positive weak solution of \eqref{eq:1} such that
  \begin{math}
    \lim_{|x|\to\infty}v(x)=0.
  \end{math}
  Then $v\equiv 0$.
\end{theorem}
If $p\ge d$, the conclusion of the theorem is valid without any
restrictions on the boundary conditions or regularity of $\Omega$ due to
a result in \cite{MR2807111}; see Section~\ref{section:proofs}. If
$1<p<d$, then under some additional assumptions on $v$ we can remove the
assumption that $\partial\Omega$ is Lipschitz. The condition is that $v$
has a trace in some weak sense which is in $L^p(\Gamma_2)$. Such a
condition is satisfied in the setting discussed in
\cite{article:non-linear-semigroups,MR1988110,MR1650081,MR2538601}.

The proof of Theorem~\ref{thm:1} relies on the asymptotic decay
estimates for positive $p$-harmonic functions on exterior domains as
stated in~\eqref{eq:6}. We give a simple alternative proof of such
estimates in case $p=2$ and $d\geq 3$ in Section~\ref{section:case-p-2}.

If $p\geq d$, then there are two alternatives for a positive
$p$-harmonic function $v$: Either $v$ is bounded in a neighbourhood of
infinity and has a limit as $|x|\to\infty$, or $v\sim \mu_{p}$ near
infinity, that is,
\begin{displaymath}
  \lim_{\abs{x}\to\infty}\frac{v(x)}{\mu_{p}(x)}=c
\end{displaymath}
for some constant $c>0$; see \cite[Theorem~2.3]{MR2807111}. In the first
case, if $v>0$, then the limit is strictly positive; see
\cite[Lemma~A.2]{MR2807111}. See also the related work
in~\cite{MR2997363}. As an example let $\Omega:=\mathbb
R^d\setminus \overline{B}_1$ and set $v_{p}:=\mu_{p}+1$ if $p=d$ and
$v_{p}:=\mu_{p}$ if $p>d$. Then $v_{p}$ is a positive $p$-harmonic
function on $\Omega$ satisfying zero Robin boundary conditions
\begin{displaymath}
  |\nabla v|^{p-2}
  \frac{\partial v}{\partial \nu}+|v|^{p-2}v=0\qquad\text{on
    $\partial\Omega$.}
\end{displaymath}
Similarly, $w_{p}:=\mu_{p}$ if $p=d$ and $w_{p}:=\mu_{p}-1$
if $p>d$ satisfies zero Dirichlet boundary conditions on
$\partial\Omega$ and is a positive unbounded $p$-harmonic function on
$\Omega$.

Our second main result is a Liouville theorem for $p$-harmonic functions
on exterior domains with zero Neumann boundary conditions, that is, the
case $\Gamma_2=\partial\Omega$ and $h\equiv 0$.
\begin{theorem}
  \label{thm:2}
  Let $\Omega$ be an exterior domain with no regularity assumption on
  $\partial\Omega$. Suppose that $v$ is a weak solution of \eqref{eq:1}
  on $\Omega$ that is bounded from below or from above. Moreover, assume
  that $v$ satisfies homogeneous Neumann boundary conditions, that is,
  $h(x,v)\equiv 0$ and $\Gamma_2=\partial\Omega$. If $1<p<d$, then $v$
  is constant.  If $p\ge d$, then $v$ is either constant or
  $v\sim\pm\mu_{p}$ near infinity.
\end{theorem}

The proofs of the theorems are based on a general criterion for
Liouville type theorems established in
Section~\ref{sec:general-criterion}. We fully prove the two Theorems in
Section~\ref{section:proofs}.

There is an intimate relationship between Liouville-type theorems and
pointwise a priori estimates of solutions of boundary value
problems. On the one hand, Liouville's theorem for some semi-linear
equations on $\R^{d}$ can be seen as a corollary of pointwise a priori
estimates; see \cite[Lemma 1]{MR1658565}. On the other hand, Liouville's
theorem can be used to derive universal upper bounds for positive
solutions on bounded domains. These connections were outlined in
\cite[p.82]{MR1946918} and recently revisited in \cite{MR2350853}. More
precisely, it is shown in \cite[p.556]{MR2350853} that Liouville's
theorem and universal boundedness theorems are equivalent for
semi-linear equations and systems of Lane-Emden type; see also
\cite{MR859333}. This relationship becomes again apparent in this
paper. This article was motivated by application to domain perturbation
problems for semi-linear elliptic boundary value problems on domains
with shrinking holes; see \cite{article:Robin-domain-perturabition}.

\medskip
\noindent \textbf{Acknowledgement. }We thank the referees for the careful
reading and pointing out some mistakes in an earlier version of the
manuskript. Also we thank for the suggestion to state
Lemma~\ref{lem:1} explicitely.

\section{Estimates near infinity in the linear case}
\label{section:case-p-2}
In this section we establish pointwise decay estimates for semi-bounded
harmonic functions on an exterior domain $\Omega\subseteq\R^{d}$ when
$d\geq 3$. The result is a special case of estimates proved in
\cite{MR2424533}, but it seems appropriate to provide a much shorter
proof in the linear case.
\begin{proposition}
  \label{prop:decay}
  Let $v$ be harmonic on the exterior domain $\Omega\subseteq\mathbb
  R^d$, $d\geq 3$. Further assume that $v$ is bounded from below or from
  above. Then $b:=\lim_{\abs{x}\to\infty}v(x)$ exists. Moreover, there
  exist positive constants $r_0$ and $C_1$ such that
  \begin{equation}
    \label{eq:3}
  \abs{v(x)-b}\leq C_{1}\,\abs{x}^{2-d}
  \end{equation}
  for all $\abs{x}\ge r_{0}$.
\end{proposition}
To prove the proposition let $v$ be a harmonic function on the exterior
domain $\Omega$, and suppose that $v$ is bounded form below or from
above.  Since by assumption $\overline{\Omega}^{c}$ is bounded, by
translating and rescaling if necessary, we can assume without loss of
generality that $\overline{\Omega}^{c}$ is contained in the unit ball
$B_{1}$, and that $0\in \overline{\Omega}^{c}$. Furthermore, without
loss of generality we can consider non-negative harmonic functions on
$\Omega$. Indeed, if $v$ is bounded from below we consider $v+\inf v\geq
0$, and if $v$ is bounded from above we consider $-v+\sup v\geq 0$.

If $v$ is positive and harmonic on $\Omega$, then in particular $v$ is
positive and harmonic on $\overline{B}_{1}^{c}$. Hence, the Kelvin
transform $K[v]$ of $v$ given by
\begin{math}
K[v](x):= \abs{x}^{2-d}v(x/\abs{x}^{2})
\end{math}
for $x\in B_{1}\setminus\{0\}$ is positive and harmonic on
$B_{1}\setminus\{0\}$; see \cite[Theorem~4.7]{MR1805196}.  By B\^ocher's
theorem there exist a harmonic function $w$ on $B_{1}$ and a constant
$b\ge0$ such that
\begin{displaymath}
  K[v](x)=w(x)+b\,\abs{x}^{2-d}\qquad\text{or}\qquad K[v-b](x)=w(x)
\end{displaymath}
for every $x\in B_{1}\setminus\{0\}$ see \cite[Theorem~3.9]{MR1805196}.
Applying the Kelvin transform again yields
\begin{equation}
  \label{eq:4}
  v(x)-b=K[w](x)=\abs{x}^{2-d}\,w(x/\abs{x}^{2})
\end{equation}
for every $x\in \overline{B}_{1}^{c}$. Note that $w(x/\abs{x}^{2})\to
w(0)$ as $\abs{x}\to\infty$.  Hence \eqref{eq:4} implies the existence
of constants $C_1,r_{0}>0$ such that \eqref{eq:3} holds whenever
$\abs{x}>r_0$.

\section{A general criterion for Liouville type theorems}
\label{sec:general-criterion}
The proofs of the main theorems are based on a general criterion showing
that a function satisfying suitable integral conditions is constant.  We
generalise an idea from \cite[Lemma 2.1]{MR2410737}. Similar ideas were
used for instance in \cite{MR1849197,MR1946918} or in
\cite[Theorem~19.8]{MR2494977} for $p=2$.

\begin{proposition}
  \label{prop:1}
  Let $\varphi\in C^{\infty}_{c}(\mathbb R^{d})$ such that $0\le \varphi
  \le 1$ on $\R^{d}$, $\varphi\equiv 1$ on $\overline{B}(0,1)$ and with
  support contained in $\overline{B}(0,2)$. For $x\in\mathbb R^d$ and
  $r>0$ let $\varphi_{r}(x):=\varphi(x/r)$. Let $v\in
  W^{1,p}_{\loc}(\Omega)$ and suppose that there exist constants
  $b\in\mathbb R$ and $C_0,C_1,r_0>0$ such that
  \begin{multline}
    \label{ineq:p-integral-Br-Kc}
    \int_{\Omega\cap B_{2r}}\abs{\nabla v}^{p}\,\varphi^{p}_{r}\,\dx\\
    \leq \frac{C_{0}}{r} \left(\int_{(\Omega\cap B_{2r})\setminus
        B_{r}} \abs{\nabla v}^{p}\,\varphi_{r}^{p}\,\dx
    \right)^{(p-1)/p}\, \left(\int_{(\Omega \cap B_{2r})\setminus B_{r}}
      \abs{v-b}^{p}\,\dx\right)^{1/p}
  \end{multline}
  and
  \begin{equation}\label{p-integral:bound}
    \frac{1}{r^{p}}\int_{(\Omega\cap B_{2r})\setminus B_{r}}
    \abs{v-b}^{p}\,\dx
    \le C_1
  \end{equation}
  for all $r>r_0$.  Then $v$ is constant.
\end{proposition}

The above proposition is a direct consequence of the following stronger
result. To prove Proposition~\ref{prop:1}, we set $C:=C_{0}C_{1}^{1/p}$ and
$\delta=(p-1)/p$ in the lemma below.  Then inequality~\eqref{eq:8} follows
from~\eqref{ineq:p-integral-Br-Kc} and~\eqref{p-integral:bound}.

\begin{lemma}
  \label{lem:1}
  Let $\varphi$ and $\varphi_{r}$ be as in Proposition~\ref{prop:1}. Let $v\in
  W^{1,p}_{\loc}(\Omega)$ and suppose that there exist constants
  $C$, $r_0>0$ and $\delta\in (0,1)$ such that
  \begin{equation}
    \label{eq:8}
    \int_{\Omega\cap B_{2r}}\abs{\nabla v}^{p}\,\varphi^{p}_{r}\,\dx
    \leq C \left(\int_{(\Omega\cap B_{2r})\setminus
        B_{r}} \abs{\nabla v}^{p}\,\varphi_{r}^{p}\,\dx
    \right)^{\delta}
  \end{equation}
  for all $r>r_0$.  Then $v$ is constant.
\end{lemma}

\begin{proof}
  In a first step we  show that $\nabla v\in
  L^{p}(\Omega)^{d}$. In a second step we then prove that $\nabla v\in
  L^{p}(\Omega)^{d}$ and~\eqref{eq:8} imply that $\nabla
  v=0$. As $\Omega$ is assumed to be connected we can apply \cite[Lemma
  7.7]{MR1814364} to conclude that $v$ is constant on $\Omega$.

  (i) We first show that $\nabla v\in L^{p}(\Omega)^{d}$. If $\nabla
  v=0$, then there is nothing to show, so assume that $\nabla v\neq
  0$. By possibly increasing $r_0$ we can assume that
  \begin{displaymath}
    \int_{\Omega\cap B_{2r}}\abs{\nabla v}^{p}\varphi^{p}_{r}\,\dr>0
  \end{displaymath}
  for all $r>r_0$. Rearranging inequality~\eqref{eq:8} and using that
  $\delta<1$ yields
  \begin{displaymath}
    \int_{\Omega}\abs{\nabla v}^{p}\varphi^{p}_{r}\,\dx
    =\int_{\Omega\cap B_{2r}}\abs{\nabla v}^{p}\varphi^{p}_{r}\,\dx
    \leq C^{1/(1-\delta)}
  \end{displaymath}
  for all $r>r_0$. Note that $\varphi_{r}^{p}\rightarrow 1_{\R^{d}}$
  pointwise and monotonically increasing as $r\to\infty$.  Hence,
  the monotone convergence theorem implies that
  \begin{equation}
    \label{eq:5}
    \int_\Omega\abs{\nabla v}^{p}\,\dx
    =\lim_{r\to\infty}\int_{\Omega\cap B_{2r}}\abs{\nabla v}^{p}\varphi^{p}_{r}\,\dx
    \le C^{1/(1-\delta)}<\infty.
  \end{equation}
  In particular $\nabla v\in L^p(\Omega)^d$ as claimed.

  (ii) Assuming that $\nabla v\in L^p(\Omega)^d$ we now show that
  $\nabla v=0$. We can rewrite~\eqref{eq:8} in the form
  \begin{displaymath}
    \int_{\Omega\cap B_{2r}}\abs{\nabla v}^{p}\,\varphi^{p}_{r}\,\dx
    \leq C \left(\int_{\Omega}\abs{\nabla v}^{p}\,\dx
      - \int_{\Omega\cap B_{r}}\abs{\nabla v}^{p}\,\varphi_{r}^{p}\,\dx
    \right)^{\delta}.
  \end{displaymath}
  Letting $r\to\infty$, making use of \eqref{eq:5} and the fact that
  $\delta>0$, we deduce that $\|\nabla v\|_p\leq 0$, that is,
  $\norm{\nabla v}_p=0$.
\end{proof}
\begin{remark}
  \label{rem:prop1}
  Suppose that $v\in W^{1,p}_{\loc}(\Omega)$ satisfies
  inequality~\eqref{ineq:p-integral-Br-Kc}, and that there exists
  $r_{0}>0$ such that $v\in L^{\infty}(\Omega\cap B_{r_{0}}^{c})$.
  Then, for every $b\in\mathbb R$
  \begin{displaymath}
    \frac{1}{r^{p}}\int_{B_{2r}\setminus B_{r}} \abs{v-b}^{p}\,\dx
    \leq \frac{\norm{v-b}_\infty^p}{r^{p}}\,\int_{B_{2r}\setminus B_{r}} 1\,\dx
    \leq \frac{\omega_d}{d}(2^d-1)\norm{v-b}_\infty^pr^{d-p}
  \end{displaymath}
  for all $r\ge r_{0}$, where $\norm{v-b}_\infty
  :=\norm{v-b}_{L^\infty(B_{r_0}^c)}$ and $\omega_d$ is the surface area
  of the unit sphere in $\mathbb R^d$. If $p\ge d$, then
  Proposition~\ref{prop:1} implies that $v$ is constant.
\end{remark}
We next show that weak solutions of \eqref{eq:1} satisfy
\eqref{ineq:p-integral-Br-Kc}. Before we can do that we want to state
our precise assumptions and give a definition of weak solutions of
boundary-value problem~\eqref{eq:1}.

\begin{assumption}
  \label{ass:1}
  By assumption, an exterior domain $\Omega$ as defined in the
  introduction is an open connected set such that $\Omega^c$ is
  compact. In particular $\partial\Omega$ is compact. Thus there exists
  $r_{0}>0$ such that $\partial\Omega\subseteq B_r:=B(0,r)$ for all
  $r\ge r_{0}$. We consider solutions of \eqref{eq:1} that lie in
  \begin{displaymath}
    V^{1,p}(\Omega):=\Bigl\{v\in W_{\loc}^{1,p}(\Omega)\colon
    \text{$v\in W^{1,p}(\Omega\cap B_r)$ for all $r>r_0$}\Bigr\}.
  \end{displaymath}
  For simplicity we now assume that $\partial\Omega$ is
  Lipschitz. We assume that $\Gamma_1,\Gamma_2$ are disjoint subsets of
  $\partial\Omega$ such that $\Gamma_1$ is closed and
  $\Gamma_1\cup\Gamma_2=\partial\Omega$. We let
  $V^{1,p}_{\Gamma_1}(\Omega)$ be the closure of the vector space
  \begin{displaymath}
    \Bigl\{v\in V^{1,p}(\Omega)\colon
    \text{$v=0$ in a neighbourhood of $\Gamma_1$}\Bigr\}
  \end{displaymath}
  in $V^{1,p}(\Omega)$. If $h(x,v)\equiv 0$ no regularity assumption on
  $\partial\Omega$ is needed.
\end{assumption}
We use the space $V^{1,p}(\Omega)$ because we do not want to assume that
the solutions of \eqref{eq:1} are in $L^p(\Omega)$.
\begin{definition}
  \label{def:wsol}
  We say that a function $v$ is a \emph{weak solution} of the boundary
  value problem~\eqref{eq:1} on $\Omega$ if $v\in
  V_{\Gamma_1}^{1,p}(\Omega)$ and
  \begin{equation}\label{ineq:weak-formulation}
    \int_{\Omega}\abs{\nabla v}^{p-2}\nabla v\nabla \varphi\,\dx
    +\int_{\Gamma_2} h(x,v)\,\varphi\,\dH = 0
  \end{equation}
  for every $\varphi\in V^{1,p}_{\Gamma_1}(\Omega)$ with
  $\supp(\varphi)\subseteq B_r$.
\end{definition}

The above definitions have to be modified in an obvious manner for
non-smooth domains. In particular, when using the setting from
\cite{MR1988110,MR1650081,MR2538601} we require that $v$ is in the
Maz'ya space $W_{p,p}^1(\Omega\cap B_r,\partial\Omega)$ for all $r$
large enough.

If $\Omega$ admits the divergence theorem and the solution $v$ is smooth
enough, then an integration by parts shows that $v$ is a weak solution
of~\eqref{eq:1} if and only if $v$ satisfies~\eqref{eq:1} in a classical
sense.  We next show that positive solutions of \eqref{eq:1} satisfy
\eqref{ineq:p-integral-Br-Kc}.

\begin{proposition}
  \label{prop:2}
  Let Assumption~\ref{ass:1} be satisfied and let $\varphi_{r}$ be as in
  Proposition~\ref{prop:1}, and $r_0>0$ such that $\Omega^c\subseteq
  B_{r_0}$. Suppose that \eqref{assum:h} is satisfied and that $v$ is a
  weak solution of \eqref{eq:1}. Then,
  inequality~\eqref{ineq:p-integral-Br-Kc} holds with $b=0$. In the case
  of homogeneous Neumann boundary conditions, that is, if $h(x,v)\equiv
  0$ and $\Gamma_2=\partial\Omega$, then every weak solution of
  \eqref{eq:1} satisfies \eqref{ineq:p-integral-Br-Kc} for every
  $b\in\mathbb R$.
\end{proposition}
\begin{proof}
  Let $r\ge r_{0}$ and let $\varphi_{r}$ be the same test-function as in
  Proposition~\ref{prop:1}. Then $v\varphi^{p}_{r}\in
  W^{1,p}_{\loc}(\Omega)$ with support in $B_{2r}$. Moreover, by
  definition of $\varphi_r$ we have $v\varphi^{p}=v$ on $\Omega\cap
  B_{r}$. Hence $v\varphi^p$ is a suitable test function to be used in
  \eqref{ineq:weak-formulation}. Using that $v$ is a weak solution of
  \eqref{eq:1} gives
  \begin{displaymath}
    \begin{split}
      0&=\int_{\Omega\cap B_{2r}}\abs{\nabla v}^{p-2}\nabla
      v\nabla(v\varphi^{p}_{r})\,\dx
      +\int_{\Gamma_2} h(x,v)\,v\,\varphi^{p}_{r}\,\dH\\
      &=\int_{\Omega\cap B_{2r}}\abs{\nabla v}^{p}\,\varphi^{p}_{r}\,\dx
      +\frac{p}{r} \int_{(\Omega\cap B_{2r})\setminus B_{r}}
      v\varphi_{r}^{p-1}\abs{\nabla
        v}^{p-2}\nabla v\nabla\varphi(\cdot/r)\,\dx\\
      &\qquad+ \int_{\Gamma_2} h(x,v)\,v\,\varphi^{p}_{r}\,\dH.
    \end{split}
  \end{displaymath}
  Rearranging this equation we arrive at
  \begin{multline*}
    \int_{\Omega\cap B_{2r}}\abs{\nabla v}^{p}\,\varphi^{p}_{r}\,\dx
    +\int_{\Gamma_2} h(x,v)\,v\,\varphi^{p}_{r}\,\dH\\
    =-\frac{p}{r} \int_{(\Omega\cap B_{2r})\setminus B_{r}}v\varphi_{r}^{p-1}\abs{\nabla
      v}^{p-2}\nabla v\nabla\varphi(\cdot/r)\,\dx.
  \end{multline*}
  By assumption \eqref{assum:h} we have $h(x,v)v\geq 0$. Setting
  $C_{0}:=p\,\norm{\nabla\varphi}_{L^{\infty}(B_{2})}$ and applying
  H\"older's inequality we obtain
  \begin{multline*}
    \int_{\Omega\cap B_{2r}}\abs{\nabla v}^{p}\,\varphi^{p}_{r}\,\dx\\
    \leq \frac{C_{0}}{r} \left(\int_{(\Omega\cap B_{2r})\setminus
        B_{r}}\abs{\nabla v}^{p}\,\varphi_{r}^{p}\,\dx
    \right)^{(p-1)/p}\, \left(\int_{(\Omega \cap B_{2r})\setminus
        B_{r}}\, \abs{v}^{p}\,\dx\right)^{1/p},
  \end{multline*}
  which is \eqref{ineq:p-integral-Br-Kc} with $b=0$. In the case of
  homogeneous Neumann boundary conditions, for every $b\in\mathbb R$,
  the function $v-b$ is another weak solution of \eqref{eq:1}. Hence we
  can replace $v$ by $v-b$ in the above calculations to obtain
  \eqref{ineq:p-integral-Br-Kc}.
\end{proof}

\begin{remark}
  Note that the above proof only uses that
  \begin{displaymath}
    0\leq\int_{\Gamma_2} h(x,v)v\varphi^{p}_{r}\,\dH
    =\int_{\Gamma_2} h(x,v)v\,\dH<\infty.
  \end{displaymath}
\end{remark}


\section{Proofs of the main theorems}
\label{section:proofs}
This section is dedicated to the proofs of Theorems~\ref{thm:1}
and~\ref{thm:2}. By rescaling, we can assume without loss of generality
that $\Omega^{c}\subseteq B_{1}$ and that $v$ is $p$-harmonic on
$\overline{B}_{1}^{c}$.

\subsection{Proof of Theorem~\ref{thm:1}.}
Assume that $1<p<d$, and that $v$ is a positive weak solution of
\eqref{eq:1} satisfying $\lim_{\abs{x}\to\infty}v(x)=0$. We need to show
that $v\equiv 0$. Due to Propositions~\ref{prop:1} and~\ref{prop:2} we
only need to show that there exists $r_{0}>0$ such that $v$ satisfies
\eqref{p-integral:bound} with $b=0$ for all $r\ge r_{0}$.  By
\eqref{eq:6} or Proposition~\ref{prop:decay} if $p=2$, there are
constants $c_1$, $c_{2}>0$ such that
\begin{displaymath}
  0\le v(x)\le c_1\,\abs{x}^{(p-d)/(p-1)}
\end{displaymath}
for every $x\in \overline{B_{2}^{c}}$. Hence,
\begin{multline}
  \label{eq:7}
  \frac{1}{r^{p}}\int_{B_{2r}\setminus B_{r}} \abs{v}^{p}\,\dx
  \le \frac{c_{1}^{p}}{r^{p}}\int_{B_{2r}\setminus B_{r}}
  \abs{x}^{p\,(p-d)/(p-1)}\,\dx\\
  = c_1^p\frac{\omega_d}{r^p}\int_{r}^{2r}s^{p\,(p-d)/(p-1)}s^{d-1}\,\ds
  = c_1^p\omega_d\, c_{2}\,r^{(p-d)/(p-1)}
\end{multline}
for all $r\geq r_0:=2$, where $\omega_d$ is the surface area of the unit
sphere in $\mathbb R^d$ and $c_{2}=\ln 2$ if $d=p^{2}$ and
$c_{2}=\frac{p-1}{d-p^2}(2^{(p^{2}-d)/(p-1)}-1)$ if $d\neq p^{2}$. As
$p<d$ we conclude that $v$ satisfies \eqref{p-integral:bound} with $b=0$
for every $r\ge 2$. As $\lim_{\abs{x}\to\infty}v(x)=0$ we conclude that
$v\equiv 0$.

If $p\geq d$, then every non-trivial positive bounded solution of
\eqref{eq:1} has a strictly positive limit as $|x|\to\infty$; see
\cite[Lemma~A.2]{MR2807111}. Because we assume that the limit is zero,
we must have $v\equiv 0$. Observe that these arguments do not make use of
the boundary conditions. This completes the proof of
Theorem~\ref{thm:1}.

\subsection{Proof of Theorem~\ref{thm:2}}
Let $v$ be a semi-bounded weak solution of problem~\eqref{eq:1} with
homogeneous Neumann boundary conditions, that is,
$\Gamma_2=\partial\Omega$ and $h(x,v)\equiv 0$. Recall also that no
regularity assumptions on $\partial\Omega$ are needed.

Note that the $p$-Laplace operator $\Delta_{p}$ is an \emph{odd}
operator, that is, $\Delta_{p}(-v)=-\Delta_{p}v$. Hence, for every $c\in
\R$, the function $c\pm v$ is another solution of
problem~\eqref{eq:1}. If $v$ is bounded from below we can therefore
replace $v$ by $v-\inf_{x\in\Omega}v(x)$, and if $v$ is bounded from
above we can replace $v$ by $\sup_{x\in\Omega}v(x)-v$. In either case we
get a new solution $v\geq 0$ with $\inf_{x\in\Omega}v(x)=0$. As before,
we also assume that $\Omega^{c}\subseteq B_1$.

If $1<p<d$, then by~\eqref{eq:6} the finite limit
$b:=\lim_{|x|\to\infty}v(x)$ exists.  By Proposition~\ref{prop:2} inequality
\eqref{ineq:p-integral-Br-Kc} is satisfied. To show that $v$ satisfies
\eqref{p-integral:bound} with $b$ just defined we repeat the calculation
\eqref{eq:7} with $v$ replaced by $v-b$, using the decay estimate from
\eqref{eq:6}. We can now apply Proposition~\ref{prop:1} to conclude that
$v$ is constant.

It remains to deal with the case $p\geq d$. Recall that
by~\cite[Theorem~2.3]{MR2807111}, every positive $p$-harmonic function
$v$ on $\Omega$ is either bounded in a neighbourhood of infinity and has
a limit $b:=\lim_{\abs{x}\to\infty}v(x)$ or $v\sim\mu_{p}$ near
infinity. In the second case the original solution considered is
asymptotically equivalent to $\pm\mu_p$ near infinity.  Assume now that
$v$ has a limit as $\abs{x}\to\infty$. To show that $v$ is constant we
first note that by Proposition~\ref{prop:2}, $v$ satisfies
\eqref{ineq:p-integral-Br-Kc} with $b=0$. As $v$ is bounded in a
neighbourhood of infinity and since $p\geq d$, Remark~\ref{rem:prop1}
implies that $v$ satisfies \eqref{p-integral:bound}. Hence by
Proposition~\ref{prop:1}, $v$ is constant. This completes the proof of
Theorem~\ref{thm:2}.



\providecommand{\bysame}{\leavevmode\hbox to3em{\hrulefill}\thinspace}

\end{document}